\documentclass[a4paper,12pt]{amsart}

\usepackage{amsfonts}
\usepackage{amsmath}
\usepackage{amssymb}

\usepackage{mathrsfs}
\usepackage[colorlinks]{hyperref}

\numberwithin{equation}{section}
\theoremstyle{plain}
\newtheorem{thm}{Theorem}[section]
\newtheorem{prop}[thm]{Proposition}

\theoremstyle{definition}
\newtheorem{defi}[thm]{Definition}
\newtheorem{rem}[thm]{Remark}
\newtheorem{ex}[thm]{Example}

\def\H{{\mathcal L}}
\def\Sp{{\mathcal H}}
\def\s{\ast}
\def\p#1{{\left({#1}\right)}}
\def\ind{\mathbb N}

\def\e[#1]{{\textrm{e}}^{#1}}
\def\Dis{{\mathrm H}^{-\infty}}
\def\Cs{{\mathrm H}}
\def\Rn{{\mathbb R}^n}

\begin{document}

\title[On nonlinear damped wave equations for positive operators]
{On nonlinear damped wave equations for positive operators. I. Discrete spectrum}

\author[Michael Ruzhansky]{Michael Ruzhansky}
\address{
  Michael Ruzhansky:
  \endgraf
  Department of Mathematics
  \endgraf
  Imperial College London
  \endgraf
  180 Queen's Gate, London, SW7 2AZ
  \endgraf
  United Kingdom
  \endgraf
  {\it E-mail address} {\rm m.ruzhansky@imperial.ac.uk}
  }
\author[Niyaz Tokmagambetov]{Niyaz Tokmagambetov}
\address{
  Niyaz Tokmagambetov:
  \endgraf
  al--Farabi Kazakh National University
  \endgraf
  71 al--Farabi ave., Almaty, 050040
  \endgraf
  Kazakhstan,
  \endgraf
   and
  \endgraf
  Institute of Mathematics and Mathematical Modeling
  \endgraf
  125 Pushkin street, Almaty, 050010
  \endgraf
  Kazakhstan
  \endgraf
  {\it E-mail address} {\rm tokmagambetov@math.kz}
 }

\thanks{The authors were supported in parts by the EPSRC
grants EP/K039407/1 and EP/R003025/1, by the Leverhulme Grants RPG-2014-02 and RPG-2017-151, as well as by the MESRK grants "Non--Harmonic Analysis and its applications" and "Non-Linear equations on graded Lie groups" (2018--2020 years) of the Committee of Science, Ministry of Education and Science of the Republic of Kazakhstan.
No new data was collected or generated during the course of research.}

\date{\today}


\begin{abstract}
In this paper we study a Cauchy problem for the nonlinear damped wave equations for a general positive operator with discrete spectrum. We derive the exponential in time decay of solutions to the linear problem with decay rate depending on the interplay between the bottom of the operator's spectrum and the mass term. Consequently, we prove global in time well-posedness results for semilinear and for more general nonlinear equations with small data. Examples are given for nonlinear damped wave equations for the harmonic oscillator, for the twisted Laplacian (Landau Hamiltonian), and for the Laplacians on compact manifolds.
\end{abstract}

\maketitle

\section{Introduction}

This work is devoted to the analysis of nonlinear damped wave equations for positive operators acting in Hilbert spaces. More precisely, for a densely defined positive operator $\H$ in a separable Hilbert space $\Sp$ we consider the Cauchy problem
\begin{equation}\label{EQ: NoL-01}
\left\{ \begin{split}
\partial_{t}^{2}u(t)+\H u(t)+b\partial_{t}u(t)+m u(t)&=F(u, \partial_{t}u, \H^{1/2}u), \quad t>0,\\
u(0)&=u_{0}\in\Sp, \\
\partial_{t}u(0)&=u_{1}\in\Sp,
\end{split}
\right.
\end{equation}
with the damping term determined by $b>0$ and mass $m\in\mathbb R$.
The main assumption in this paper is that the operator $\H$ has a discrete spectrum
%with eigenvalues of finite multiplicities\footnote{\bf finite multiplicities - needed or not? I think not needed here?!}
and that the corresponding eigenvectors form an orthonormal basis in $\Sp$.

The main examples of interest for us would be the harmonic oscillator on $\Sp=L^2(\Rn)$:
\begin{equation}\label{EQ:ho}
\H:=-\Delta+|x|^{2}, \,\,\, x\in\mathbb R^{n},
\end{equation}
and the Laplacians, or more general positive elliptic pseudo-differential operators, on $\Sp=L^2(M)$ for compact manifolds $M$, with or without boundary. Of course there are numerous other examples that are covered by this setting, for example the twisted Laplacian (Landau Hamiltonian)
on $\mathbb C^{n}$ given by
$$
\H=\sum_{j=1}^{n}(Z_{j}\bar{Z}_{j}+\bar{Z}_{j}Z_{j}),
$$
with $Z_{j}=\frac{\partial}{\partial z_{j}}+\frac{1}{2}\bar{z}_{j}$ and $\bar{Z}_{j}=-\frac{\partial}{\partial \bar{z}_{j}}+\frac{1}{2}z_{j}$, see Example \ref{EX:LH}, where we also derive the the Gagliardo-Nirenberg inequality for it.

The other important situation occurs when the spectrum of $\H$ is continuous. In that case the analysis relies on rather different methods and this problem will be addressed in the subsequent paper.

The analysis of linear and nonlinear damped wave equations has a long history. In papers \cite{M76, Wahl70} the authors first considered these kind of problems for the Laplacian on $\Rn$. We refer to papers \cite{HKN04, HKN06, HL92, HO04, I04, IMN04, IT05, KU13, Kh13, N04, N03, SW07, O03, O06, R90} in $\Rn$ dealing with damped wave equations under different assumptions, and references therein, where authors study the global solvability of the Cauchy problems for nonlinear wave equations for the Laplace operator with the dissipative term. Also, see \cite{K00, RTY11, W14} for some more abstract settings. For even more references, we refer to a recent survey \cite{IIW17}.
Time-dependent dissipation has been also considered, see e.g. \cite{W06} for regular and \cite{GR15, RT16a} for irregular dissipation in linear problems, respectively.
The global framework for the Fourier analysis generated by a densely defined operator $\H$ on $L^2(M)$ for manifolds $M$ with or without boundary was developed in \cite{RT16, RT16b}.

In Section \ref{SEC:linear} we consider the linear equations and derive the exponential time decay for their solutions. This is done by using the Fourier analysis adapted to the operator $\H$, elements of which we review in the process of the proof. The exponential decay plays a crucial role in the further analysis, in particular allowing the handling of the nonlinear equations to rely mostly on the analysis in Sobolev spaces over $\H$. Such decay is achieved by the fact that the operator $\H$ has a discrete positive spectrum. In the case of continuous spectrum, more delicate $L^p$- methods are needed, and these will appear elsewhere in the subsequent analysis for that setting.

Partial differential equations in general Hilbert (and also Banach) spaces have been considered in many papers as well, both linear and nonlinear. For example, see \cite{EFNT94} for an extensive analysis in terms of the dynamical systems behaviour, and \cite{Zua90} for related analysis.
Linear wave equations in Hilbert spaces with irregular coefficients have been recently considered by the authors in \cite{RT17}.

In Section \ref{SEC:semlinear} we consider the case of semilinear damped wave equations of the form
\begin{equation}\label{EQ: NoL-01i}
\left\{ \begin{split}
\partial_{t}^{2}u(t)+\H u(t)+b\partial_{t}u(t)+m u(t)&=f(u), \quad t>0,\\
u(0)&=u_{0}\in\Sp, \\
\partial_{t}u(0)&=u_{1}\in\Sp,
\end{split}
\right.
\end{equation}
under the assumption that
$f$ satisfies the properties
\begin{equation} \label{PR: f-007i}
\left\{
\begin{split}
f(0) & =0, \\
|f(u)-f(v)| & \leq C (|u|^{p-1}+|v|^{p-1})|u-v|,
\end{split}
\right.
\end{equation}
for $u, v\in\mathbb R$.
If $\Sp=L^2$, then an example of $f$ satisfying \eqref{PR: f-007i} is given by
$$f(u)=\mu |u|^{p-1} u,$$
for $p>1$ and $\mu\in\mathbb R$ or, more generally, by
differentiable functions $f$ such that $$|f'(u)|\leq C|u|^{p-1}.$$

In Section \ref{S: n+2} we consider a general case, namely, we deal with the nonlinear equation \eqref{EQ: NoL-01} for general nonlinearity $$F=F(u, \partial_{t}u, \H^{1/2} u)$$ satisfying an analogue of the Gagliardo-Nirenberg inequalities in the Sobolev space associated to $\H$.
This condition is formulated in \eqref{PR: f-02} and some examples for it are given in \eqref{EQ:exnon}.

In Section \ref{SEC:higher} we consider more general nonlinearities of the form
\begin{equation}\label{EQ:F-gen-i}
F_l=F_l(u,\{\partial^j u\}_{j=1}^l, \{{\mathcal L}^{j/2} u\}_{j=1}^{l}).
\end{equation}
for $F_l:{\mathbb C}^{2l+1}\to {\mathbb C}$, for any $l\in\mathbb N$. The proof of the global in time well-posedness in this case is an extension of the proof in Section \ref{S: n+2}, so we only very briefly indicate the differences there. A different feature here is that the smallness is required in higher regularity Sobolev spaces, but  only to make sense of the higher order derivatives entering the nonlinearity $F_l$.

From the physical point of view it is natural to assume that $b>0$ and $m\geq 0$. However, from the point of view of the well-posedness we may allow $m$ to be negative. In this case, there appears an interplay between $b, m$ and the bottom $\lambda_0$ of the spectrum of $\H$.
The global in time decay properties of solutions to wave equation with negative mass in $\Rn$ were derived in \cite{RSm}.

The inclusion of the mass term does allow us to derive certain results even in the case when the bottom of the spectrum of $\H$ is zero. For example, when $\Sp=L^2(M)$ for a compact manifold $M$ without boundary, and $\H$ being the positive Laplacian on $M$, the operator $\H$ has a zero eigenvalue. In this case, if the Cauchy data are constant, in the case of $m=0$ the solution to the linear problem allows also constants, so that there is no decay in time and no dispersion even for $b>0$. To avoid this kind of (trivial) problems, it will be convenient to assume that $\lambda_0+m>0$, with $\lambda_0\geq 0$ and $m\in\mathbb R$. Otherwise, to summarise and collect our assumptions for this paper, we will be assuming throughout that

\begin{center}
\em The positive operator $\H$ has a discrete spectrum $\{\lambda_j\}_{j\in\mathbb N}$ with
$\lambda_{0}:=\inf\limits_{j\in\mathbb N}\lambda_j\geq 0$,
and the corresponding eigenvectors form an orthonormal basis in $\Sp$.

Moreover, we assume that $b>0$, $m\in\mathbb R$, and $\lambda_0+m>0$.
\end{center}

For example, in our setting, assuming that $\lambda_0+m>0$, for appropriate indices $\alpha, \beta$, we then have the estimates
\begin{equation*}\label{Est-0-01i}
\|\partial_{t}^{\alpha}\H^{\beta}u(t)\|_{\Sp}\lesssim\e[-\frac{b}{2}t] \,\,
(\|u_{0}\|_{H_{\H}^{\alpha+2\beta}}+\|u_{1}\|_{H_{\H}^{\alpha-1+2\beta}}),
\; \textrm{ for } 0<b< 2\sqrt{\lambda_{0}+m},
\end{equation*}
\begin{equation*}\label{Est-0-01bi}
\|\partial_{t}^{\alpha}\H^{\beta}u(t)\|_{\Sp}\lesssim (1+t)\e[-\frac{b}{2}t] (\|u_{0}\|_{H_{\H}^{\alpha+2\beta}}+\|u_{1}\|_{H_{\H}^{\alpha-1+2\beta}}),
\; \textrm{ for } b=2\sqrt{\lambda_{0}+m},
\end{equation*}
and
\begin{equation*}\label{Est-0-02i}
\|\partial_{t}^{\alpha}\H^{\beta}u(t)\|_{\Sp}\lesssim\e[-(\frac{b}{2}-\sqrt{\frac{b^{2}}{4}-\lambda_{0}-m})t] (\|u_{0}\|_{H_{\H}^{\alpha+2\beta}}+\|u_{1}\|_{H_{\H}^{\alpha-1+2\beta}}),
\; \textrm{ for } 2\sqrt{\lambda_{0}+m}<b,
\end{equation*}
for solutions of linear and nonlinear equations, modulo small modifications -- see the exact statements later on, e.g. in Proposition \ref{LEM: Est-01}.

\smallskip
Throughout this paper we will use the notation $\lesssim$ to not write constants (which are not depending on the main parameters) in estimates. We also assume $\mathbb N=\{1,2,\ldots\}$ and $\mathbb N_0=\mathbb N\cup\{0\}.$

\smallskip
The authors would like to thank both referees for the useful and constructive comments.

\section{Dissipative wave equation}
\label{SEC:linear}

In this section we derive energy estimates for the linear damped wave equation
\begin{equation}\label{CPa-01}
\left\{ \begin{split}
\partial_{t}^{2}u(t)+\H u(t)+b\partial_{t}u(t)+m u(t)&=0, \quad t>0, \\
u(0)&=u_{0}\in\Sp, \\
\partial_{t}u(0)&=u_{1}\in\Sp,
\end{split}
\right.
\end{equation}
for some dissipation constant $b>0$ and some mass constant $m$.

The time decay rates will depend on the following parameter associated with $\H$. Let $\{\lambda_j\}_{j=1}^{\infty}$ be the set of eigenvalues of $\H$. Since $\H$ is a positive operator, all eigenvalues are also positive. Then we call
\begin{equation}\label{EQ: Eigen-Parameter}
\lambda_{0}:=\inf\limits_{j\in\mathbb N}\lambda_j
\end{equation}
the bottom of the spectrum of $\H$.
In this paper we make the only assumption that the operator is positive, i.e. that 
\begin{equation}\label{EQ:lam0}
\lambda_{0}\geq 0.
\end{equation}

To obtain the time decay rate for solutions $u(t)$ of \eqref{CPa-01} we first derive the representation of solutions for \eqref{CPa-01} based on the suitable Fourier analysis adapted to the operator $\H$.

For this, we first recall the necessary elements of the global Fourier analysis
that has been developed in \cite{RT16} (also see \cite{RT16b}, and its applications to the spectral properties of operators in \cite{DRT16}). Since the operator $\H$ is self-adjoint, the construction of
\cite{RT16} is considerably simplified. We now give its brief review adapting it to the present setting.

Let $\Cs_{\H}^{\infty}:={\rm Dom}({\H}^{\infty})$ be the space of test
functions for ${\H}$ which we define as
$$
{\rm Dom}({\H}^{\infty}):=\bigcap_{k=1}^{\infty}{\rm Dom}({\H}^{k}),
$$
where ${\rm Dom}({\H}^{k})$ is the domain of the operator ${\H}^{k}$, in turn defined as
$$
{\rm Dom}({\H}^{k}):=\{f\in\Sp: \,\,\, {\H}^{j}f\in {\rm Dom}({\H}), \,\,\, j=0,
\,1, \, 2, \ldots,
k-1\}.
$$
The Fr\'echet topology of $\Cs_{{\H}}^{\infty}$ is given by the family of semi-norms
\begin{equation}\label{EQ:L-top}
\|\varphi\|_{\Cs^{k}_{{\H}}}:=\max_{j\leq k}
\|{\H}^{j}\varphi\|_{\Sp}, \quad k\in\mathbb N_0,
\; \varphi\in\Cs_{{\H}}^{\infty}.
\end{equation}
The space $$\Dis_{{\H}}:={\mathscr L}(\Cs_{\H}^{\infty},
\mathbb C)$$ of linear continuous functionals on
$\Cs_{\H}^{\infty}$ is called the space of
${\H}$-distributions.
For
$w\in\Dis_{{\H}}$ and $\varphi\in\Cs_{\H}^{\infty}$,
we shall write
$$
w(\varphi)=\langle w, \varphi\rangle.
$$
For any $\psi\in\Cs_{\H}^{\infty}$, the functional
$$
\Cs_{\H}^{\infty}\ni \varphi\mapsto (\varphi,\psi)
$$
is an ${\H}$-distribution, which gives an embedding $\psi\in\Cs_{{\H}}^{\infty}\hookrightarrow\Dis_{\H}$.
Let $\mathcal S(\ind)$ denote the space of rapidly decaying
functions $\varphi:\ind\rightarrow\mathbb C$. That is,
$\varphi\in\mathcal S(\ind)$ if for any $N<\infty$ there
exists a constant $C_{\varphi, N}$ such that
$$
|\varphi(\xi)|\leq C_{\varphi, m}\langle\xi\rangle^{-N}
$$
holds for all $\xi\in\ind$, where we denote
$$\langle\xi\rangle:=(1+|\lambda_{\xi}|)^{1/2},$$
where $\lambda_\xi$ are the eigenvalues of $\Sp$ labelled according to multiplicities.
We denote by $e_\xi$ the corresponding eigenvectors of $\Sp$.

The topology on $\mathcal
S(\ind)$ is given by the seminorms $p_{k}$, where
$k\in\mathbb N_{0}$ and
$$
p_{k}(\varphi):=\sup_{\xi\in\ind}\langle\xi\rangle^{k}|\varphi(\xi)|.
$$

We now define the $\H$-Fourier transform on $\Cs_{\H}^{\infty}$ as the mapping
$$
(\mathcal F_{\H}f)(\xi)=(f\mapsto\widehat{f}):
\Cs_{\H}^{\infty}\rightarrow\mathcal S(\ind)
$$
by the formula
\begin{equation}
\label{FourierTr}
\widehat{f}(\xi):=(\mathcal F_{\H}f)(\xi)=(f, e_{\xi}).
\end{equation}
The $\H$-Fourier transform
$\mathcal F_{\H}$ is a bijective homeomorphism from $\Cs_{{\H}}^{\infty}$ to $\mathcal
S(\ind)$.
Its inverse  $$\mathcal F_{\H}^{-1}: \mathcal S(\ind)
\rightarrow \Cs_{\H}^{\infty}$$ is given by
\begin{equation}
\label{InvFourierTr}
\mathcal F^{-1}_{{\H}}h=\sum_{\xi\in\ind} h(\xi) e_{\xi},\quad
h\in\mathcal S(\ind),
\end{equation}
so that the Fourier inversion formula becomes
\begin{equation}
\label{InvFourierTr0}
f=\sum_{\xi\in\ind} \widehat{f}(\xi)e_{\xi}
\quad \textrm{ for all } f\in\Cs_{{\H}}^{\infty}.
\end{equation}
The Plancherel's identity takes the form
\begin{equation}\label{EQ:Plancherel}
\|f\|_{\Sp}=\p{\sum_{\xi\in\ind}
|\widehat{f}(\xi)|^2}^{1/2}.
\end{equation}

\smallskip
Consequently, we can also define Sobolev spaces $H^s_\H$ associated to
$\H$. Thus, for any $s\in\mathbb R$, we set
\begin{equation}\label{EQ:HsL}
H^s_\H:=\left\{ f\in\Dis_{\H}: \H^{s/2}f\in
\Sp\right\},
\end{equation}
with the norm $\|f\|_{H^s_\H}:=\|\H^{s/2}f\|_{\Sp}$, which we understand as
\begin{equation*}\label{EQ:Hsub-norm}
\|f\|_{H^s_\H}:=\|\H^{s/2}f\|_{\Sp}:=
\p{\sum_{\xi\in\ind} \lambda_\xi^{s} |\widehat{f}(\xi)|^{2}}^{1/2}.
\end{equation*}
In particular, for $s=0$, we have $H^0_\H=\Sp.$

For $f,g\in\Sp$ the convolution $(f\ast_{\H}
g)$ was defined in \cite{RT16} by the formula
$$
f\ast_{\H} g : = \sum\limits_{\xi\in\ind}\widehat{f}(\xi) \, \widehat{g}(\xi)\,e_{\xi}.
$$
This convolution and its properties in general Hilbert spaces have been analysed in \cite{KRT17, RT16c}.
In terms of this convolution, the solution of \eqref{CPa-01} is given as
$$
u(t)=K_{0}(t)\s_{\H} u_{0}+K_{1}(t)\s_{\H} u_{1},
$$
where the $\H$-Fourier transforms $R_{i}(t, \xi)$ of $K_{i}(t)$ $(i=0,1)$ are determined from the ordinary differential equations
\begin{equation}\label{CPa-02}
\left\{ \begin{split}
\partial_{t}^{2}\widehat{u}(t, \xi)+b\partial_{t}\widehat{u}(t, \xi)+(\sigma_{\H}(\xi)+m) \widehat{u}(t, \xi)&=0, \quad t>0,\\
\widehat{u}(0, \xi)&=\widehat{u}_{0}(\xi), \\
\partial_{t}\widehat{u}(0, \xi)&=\widehat{u}_{1}(\xi),
\end{split}
\right.
\end{equation}
where $\sigma_{\H}(\xi)=\lambda_\xi$ is the symbol of the operator $\H$.
In the case $\sigma_{\H}(\xi)+m\neq b^{2}/4$ the equations \eqref{CPa-02} can be solved explicitly with their solutions given by
$$
\widehat{u}(t, \xi)=C_{0}\e[(-b/2+i\sqrt{\sigma_{\H}(\xi)+m-b^{2}/4})t]+C_{1}\e[(-b/2-i\sqrt{\sigma_{\H}(\xi)+m-b^{2}/4})t],
$$
where
$$
C_{0}=\left(\frac{b}{4i\sqrt{\sigma_{\H}(\xi)+m-b^{2}/4}}+\frac{1}{2}\right)\widehat{u}_{0}(\xi)
+\frac{1}{2i\sqrt{\sigma_{\H}(\xi)+m-b^{2}/4}}\widehat{u}_{1}(\xi),
$$
and
$$
C_{1}=\left(\frac{i b}{4\sqrt{\sigma_{\H}(\xi)+m-b^{2}/4}}+\frac{1}{2}\right)\widehat{u}_{0}(\xi)+
\frac{i}{2\sqrt{\sigma_{\H}(\xi)+m-b^{2}/4}}\widehat{u}_{1}(\xi).
$$
And, for the case $\sigma_{\H}(\xi)+m=b^{2}/4$ the equations \eqref{CPa-02} can be solved with their solutions given by
$$
\widehat{u}(t, \xi)=C_{0}\e[(-b/2)t]+C_{1} t \, \e[(-b/2)t],
$$
where
$$
C_{0}=\widehat{u}_{0}(\xi), \,\,\, C_{1}=\frac{b}{2}\widehat{u}_{0}(\xi)+\widehat{u}_{1}(\xi).
$$
Thus, for $\sigma_{\H}(\xi)+m\neq b^{2}/4$, we obtain
\begin{align*}
R_{0}(t, \xi)=\left(\frac{b}{4i\sqrt{\sigma_{\H}(\xi)+m-b^{2}/4}}+\frac{1}{2}\right)\e[(-b/2+i\sqrt{\sigma_{\H}(\xi)+m-b^{2}/4})t]\\
+\left(\frac{i b}{4\sqrt{\sigma_{\H}(\xi)+m-b^{2}/4}}+\frac{1}{2}\right)\e[(-b/2-i\sqrt{\sigma_{\H}(\xi)+m-b^{2}/4})t],
\end{align*}
and
\begin{align*}
R_{1}(t, \xi)=\frac{1}{2i\sqrt{\sigma_{\H}(\xi)+m-b^{2}/4}}\e[(-b/2+i\sqrt{\sigma_{\H}(\xi)+m-b^{2}/4})t]\\
+\frac{i}{2\sqrt{\sigma_{\H}(\xi)+m-b^{2}/4}}\e[(-b/2-i\sqrt{\sigma_{\H}(\xi)+m-b^{2}/4})t],
\end{align*}
and, for $\sigma_{\H}(\xi)+m=b^{2}/4$, we have
\begin{align*}
R_{0}(t, \xi)=\left(1+\frac{b}{2}t \right) \, \e[(-b/2)t], \,\,\, R_{1}(t, \xi)=t \, \e[(-b/2)t].
\end{align*}

Thus, using these formulae, we get
%\footnote{\bf NIYAZ: it is not clear to me whether we need to assume that $\lambda_\xi\to \infty$, or even stronger that this holds with  some rate. This would be needed e.g. to have $L^p(\Omega)\subset H^{-\infty}_\H$ but do we need it? This assumption would rule out the Landau Hamiltonian so we need to see carefully whether it is needed, In any case, we probably need to assume that $\lambda_\xi=b$ can happen for at most finitely many $\xi$, to avoid the problem of division by $0$ in the above formulae.}

\begin{prop}\label{LEM: Est-01} Let $\lambda_{0}\geq 0$ be the bottom of the spectrum of $\H$ defined by \eqref{EQ: Eigen-Parameter}. Assume that $\lambda_{0}+m>0$. Then the solution $u$ of \eqref{CPa-01} satisfies the estimates
\begin{equation}\label{Est-0-01}
\|\partial_{t}^{\alpha}\H^{\beta}u(t)\|_{\Sp}\lesssim\e[-\frac{b}{2}t] \,\,
(\|u_{0}\|_{H_{\H}^{\alpha+2\beta}}+\|u_{1}\|_{H_{\H}^{\alpha-1+2\beta}}),
\end{equation}
for $0<b< 2\sqrt{\lambda_{0}+m}$, and
\begin{equation}\label{Est-0-01b}
\|\partial_{t}^{\alpha}\H^{\beta}u(t)\|_{\Sp}\lesssim (1+t)\e[-\frac{b}{2}t] (\|u_{0}\|_{H_{\H}^{\alpha+2\beta}}+\|u_{1}\|_{H_{\H}^{\alpha-1+2\beta}}),
\end{equation}
for $b=2\sqrt{\lambda_{0}+m}$, and
\begin{equation}\label{Est-0-02}
\|\partial_{t}^{\alpha}\H^{\beta}u(t)\|_{\Sp}\lesssim\e[-(\frac{b}{2}-\sqrt{\frac{b^{2}}{4}-\lambda_{0}-m})t] (\|u_{0}\|_{H_{\H}^{\alpha+2\beta}}+\|u_{1}\|_{H_{\H}^{\alpha-1+2\beta}}),
\end{equation}
for $2\sqrt{\lambda_{0}+m}<b$,
for all $\alpha\in\mathbb N_{0}$ and $\beta\geq 0.$
\end{prop}

\begin{proof}%[Proof of Proposition \ref{LEM: Est-01}]
By taking into account the equalities
\begin{equation*}
\|\partial_{t}^{\alpha}\H^{\beta}u\|_{\Sp}=\|\mathcal F_{\H}(\partial_{t}^{\alpha}\H^{\beta}u)\|_{l^{2}},
\end{equation*}
\begin{equation*}
\mathcal F_{\H}(\partial_{t}^{\alpha}\H^{\beta}u)=\sigma_{\H}^{\beta}(\xi)\mathcal F_{\H}(\partial_{t}^{\alpha}u),
\end{equation*}
and the representations of $R_{0}(t, \xi)$ and $R_{1}(t, \xi)$, we obtain the statement of Proposition \ref{LEM: Est-01}.
\end{proof}

\begin{rem}\label{REM:positivity}
We note that we could combine the operator $\H$ and the mass term $m$ into a new operator $\H+m$. Then, the statement of Proposition \ref{LEM: Est-01} would hold under the assumption that the bottom of the spectrum $\H+m$ is $>0$, without assuming that the operator $\H$ is positive. However, we prefer to formulate it in this form since the operator $\H^{1/2}$ will appear later on in the Gagliardo-Nirenberg inequality in \eqref{Inty-G-N-01}.
\end{rem}

\section{Semilinear damped wave equation}
\label{SEC:semlinear}

In this section we consider the semilinear damped wave equation for the operator $\H$, taking the form:
\begin{equation}\label{EQ: NoL-01s}
\left\{ \begin{split}
\partial_{t}^{2}u(t)+\H u(t)+b\partial_{t}u(t)+m u(t)&=f(u), \quad t>0,\\
u(0)&=u_{0}\in\Sp, \\
\partial_{t}u(0)&=u_{1}\in\Sp.
\end{split}
\right.
\end{equation}
A typical example that we are interested in is $\Sp=L^2(\Rn)$ or $\Sp=L^2(M)$ for a compact manifold $M$, and
\begin{equation}\label{EQ:nonlinex}
f(u)=\mu |u|^{p-1} u,
\end{equation}
for $p>1$ and $\mu\in\mathbb R$.
However, we will be able to prove the global in time well-posedness for a more general class of nonlinearities $f(u)$ in abstract Hilbert spaces, satisfying the conditions \eqref{PR: f-007} in Theorem \ref{TH: 01}.

We now introduce the following notion of the Gagliardo--Nirenberg index that will be important for our global in time well-posedness result for \eqref{EQ: NoL-01s}.
We may identify our Hilbert space $\Sp$ as $\Sp=L^2(\Omega)$ for a measure space $\Omega$, so that
we can also use the scale $L^p(\Omega)$ of spaces on $\Omega$.
We can write $\|\cdot\|_{\Sp}=\|\cdot\|_2$ in this notation.

\begin{defi}[Gagliardo--Nirenberg index]
\label{DEF:GN}
We say that $p\geq 1$ is Gag\-li\-ardo--Nirenberg admissible for the operator $\H$ if
the Gagliardo--Nirenberg type inequality
\begin{equation}\label{Inty-G-N-01}
\|u\|_{2p}\leq C \|\H^{1/2} u\|_{2}^{\theta} \, \|u\|_{2}^{1-\theta}
\end{equation}
holds for some $\theta=\theta(p)\in[0, 1]$.
\end{defi}

\begin{ex}[Harmonic oscillator]
Note that for the harmonic oscillator $\H=-\Delta+|x|^{2}$ in $\mathbb R^{n}$, the following indices are Gagliardo--Nirenberg admissible, i.e. we have that \eqref{Inty-G-N-01} holds for
\begin{equation}\label{EQ:GN}
\left\{ \begin{split}
n=1 \,\, \hbox{and} \,\,
n=2: & \,\,\, 1\leq p< \infty; \\
n\geq3: & \,\,\, 1\leq p\leq \frac{n}{n-2}.
\end{split}
\right.
\end{equation}
These properties follow from the corresponding properties of the Laplacian, see the results of Nirenberg's paper \cite{N59}. In this case we also have the bottom of the spectrum of $\H$ given by $\lambda_0=n$, see Appendix \ref{APP}.
\end{ex}

\begin{ex}[Laplacian on spheres]
For $\H$ being the Laplacian on the sphere $\mathbb S^n$, with $\Sp=L^2(\mathbb S^n)$,
the Gagliardo-Nirenberg admissible indices are also given by \eqref{EQ:GN}, see e.g. \cite{Dol}.
%\footnote{\bf Niyaz: please make more examples, adding results from the papers you found, also adding them as references.}
\end{ex}

\begin{ex}[Laplacian on compact Riemannian manifolds]
More generally, for $\H$ being the Laplacian on the compact Riemannian manifold $\mathcal M$, with $\Sp=L^2(\mathcal M)$, the Gagliardo-Nirenberg admissible indices are given by
%\begin{equation}\label{EQ:GN-02}n\geq1: \,\,\, 1\leq p\leq 1+\frac{2}{n},\end{equation}for this, see \cite{B03}. Here the author consider the best-constant problem for a family of GagliardoNirenberg inequalities on a compact Riemannian manifold.
\begin{equation}\label{EQ:GN-03}
\left\{ \begin{split}
n=2: & \,\,\, 1\leq p< \infty; \\
n\geq3: & \,\,\, 1\leq p\leq \frac{n}{n-2}.
\end{split}
\right.
\end{equation}
For this, we refer to the papers of Ceccon and Montenegro \cite{CM08, CM13}. For more references, see e.g. \cite{B03, ACM15, CD16} and references therein.
\end{ex}

\begin{ex}[Landau Hamiltonian]\label{EX:LH}
If we take the twisted Laplacian on $\mathbb C^{n}$
$$
\H=\sum_{j=1}^{n}(Z_{j}\bar{Z}_{j}+\bar{Z}_{j}Z_{j}),
$$
with $Z_{j}=\frac{\partial}{\partial z_{j}}+\frac{1}{2}\bar{z}_{j}$, and $\bar{Z}_{j}=-\frac{\partial}{\partial \bar{z}_{j}}+\frac{1}{2}z_{j}$, then the Gagliardo-Nirenberg admissible indices for the twisted Laplacian (Landau Hamiltonian) $\H$ are given by
\begin{equation}\label{EQ:GN-000}
\left\{ \begin{split}
n=1: & \,\,\, 1\leq p< \infty; \\
n\geq2: & \,\,\, 1\leq p\leq \frac{n}{n-1}.
\end{split}
\right.
\end{equation}
Moreover, we have $\lambda_0=n$.  The spectrum of $\H$ is discrete but the eigenvalues have infinite multiplicities, see e.g. \cite {RS15} or \cite{RT16a}. However, as we do not make any assumption on multiplicities, this situation is covered by our setting.

\begin{proof}[Proof of \eqref{EQ:GN-000}]
It follows from the H\"{o}lder inequality that
$$
\int_{\mathbb C^{n}}|u|^{2p}dz=\int_{\mathbb C^{n}}|u|^{2ps}|u|^{2p(1-s)}dz\leq\left(\int_{\mathbb C^{n}}|u|^{\frac{2n}{n-1}}dz\right)^{ps \frac{n-1}{n}} \left(\int_{\mathbb C^{n}}|u|^{2}dz\right)^{p(1-s)},
$$
for any $s\in [0,1]$ such that
\begin{equation}\label{EQ: ind-p}
ps\frac{n-1}{n}+p(1-s)=1.
\end{equation}
Then by using the Sobolev embedding from \cite[Lemma 2.3]{RS15}, we obtain
$$
\|u\|_{L^{2p}(\mathbb C^{n})}\lesssim \|u\|_{\dot{L}_{1}^{2}(\mathbb C^{n})}^{s} \|u\|_{L^{2}(\mathbb C^{n})}^{1-s}.
$$
Finally, \eqref{EQ: ind-p} yields \eqref{EQ:GN-000}.
\end{proof}
\end{ex}

For the convenience of the reader we recall the definition of the  Sobolev spaces $H^s_\H$, $s\in\mathbb R$, associated to
$\H$:
\begin{equation*}\label{EQ:HsL-00}
H^s_\H:=\left\{ f\in H^{-\infty}_\H: \H^{s/2}f\in
L^2\right\},
\end{equation*}
with the norm $\|f\|_{H^s_\H}:=\|\H^{s/2}f\|_{L^2}.$
We also recall that $\lambda_0=\lambda_0(\H)$ denotes the bottom of the spectrum of $\H$ defined by \eqref{EQ: Eigen-Parameter}.

\begin{thm} \label{TH: 01}
Let $p>1$ be Gagliardo-Nirenberg admissible for $\H$, i.e. assume that \eqref{Inty-G-N-01} holds. Suppose that $\lambda_0\geq 0$ and $\lambda_0+m>0$. Assume that $f$ satisfies the properties
\begin{equation} \label{PR: f-007}
\left\{
\begin{split}
f(0) & =0, \\
|f(u)-f(v)| & \leq C (|u|^{p-1}+|v|^{p-1})|u-v|,
\end{split}
\right.
\end{equation}
for $u, v\in\mathbb R$. Assume that the Cauchy data $u_{0}\in H_{\H}^{1}$ and $u_{1}\in \Sp$ satisfy
\begin{equation} \label{EQ: Th-cond-01}
\|u_{0}\|_{H_{\H}^{1}}+\|u_{1}\|_{\Sp}\leq\varepsilon.
\end{equation}
Then, there exists a small positive constant $\varepsilon_{0}>0$ such that the Cauchy problem \begin{equation*}\label{EQ: NoL-02}
\left\{ \begin{split}
\partial_{t}^{2}u(t)+\H u(t)+b\partial_{t}u(t) + m u(t)&=f(u), \quad t>0,\\
u(0)&=u_{0}\in H_{\H}^{1}, \\
\partial_{t}u(0)&=u_{1}\in\Sp,
\end{split}
\right.
\end{equation*}
has a unique global solution $u\in C(\mathbb R_{+}; H_{\H}^{1})\bigcap C^{1}(\mathbb R_{+}; \Sp)$ for all $0<\varepsilon\leq\varepsilon_{0}$.

Moreover, when $0 < b < 2\sqrt{\lambda_{0}+m}$ we have
\begin{equation}\label{TH-NoL-1-01}
\|\partial_{t}^{\alpha}\H^{\beta}u(t)\|_{\Sp}\lesssim (1+t)^{1/2} \e[-\frac{b}{2}t],
\end{equation}
and when $b = 2\sqrt{\lambda_{0}+m}$ we have
\begin{equation}\label{TH-NoL-1-01}
\|\partial_{t}^{\alpha}\H^{\beta}u(t)\|_{\Sp}\lesssim (1+t)^{3/2} \e[-\frac{b}{2}t],
\end{equation}
and when $2\sqrt{\lambda_{0}+m}<b$ we have
\begin{equation}\label{TH-NoL-1-03} \|\partial_{t}^{\alpha}\H^{\beta}u(t)\|_{\Sp}\lesssim (1+t)^{1/2}\e[-(\frac{b}{2}-\sqrt{\frac{b^{2}}{4}-\lambda_{0}-m})t],
\end{equation}
%as $t\to\infty$
for $(\alpha, \beta)=(0, 0)$ and $(\alpha, \beta)=(0, 1/2)$, and $(\alpha, \beta)=(1, 0)$.
\end{thm}

As noted in the introduction, if $\Sp=L^2$, then an example of $f$ satisfying \eqref{PR: f-007} is given by \eqref{EQ:nonlinex}, i.e. by $$f(u)=\mu |u|^{p-1} u, \quad p>1,\;  \mu\in\mathbb R,$$ or by a differentiable function $f$ such that
$$|f'(u)|\leq C|u|^{p-1}.$$

\begin{proof}[Proof of Theorem \ref{TH: 01}]
Let us consider the closed subsets $Z_j$ of the space $C^{1}(\mathbb R_{+}; \,\, H^{1}_{\H})$ defined as
$$
Z_{j} :=\{u\in C^{1}(\mathbb R_{+}; \,\, H^{1}_{\H}); \,\, \|u\|_{Z_{j}}\leq L_{j}\}, \,\,\, j=1,2,3,
$$
with
\begin{align*}
\|u\|_{Z_{1}}:=\sup_{t\geq0}\{(1+t)^{-1/2}\e[\frac{b}{2}t](\|u(t, \cdot)\|_{2}+\|\partial_{t}u(t, \cdot)\|_{2}+\|\H^{1/2}u(t, \cdot)\|_{2})\},
\end{align*}
if $0<b< 2\sqrt{\lambda_{0}+m}$, and
\begin{align*}
\|u\|_{Z_{2}}:=\sup_{t\geq0}\{(1+t)^{-3/2}\e[\frac{b}{2}t](\|u(t, \cdot)\|_{2}+\|\partial_{t}u(t, \cdot)\|_{2}+\|\H^{1/2}u(t, \cdot)\|_{2})\},
\end{align*}
if $b = 2\sqrt{\lambda_{0}+m}$, and
\begin{align*}
\|u\|_{Z_{3}}:=\sup_{t\geq0}\{(1+t)^{-1/2}\e[(\frac{b}{2}-\sqrt{\frac{b^{2}}{4}-\lambda_{0}-m})t](\|u(t, \cdot)\|_{2}+\|\partial_{t}u(t, \cdot)\|_{2}+\|\H^{1/2}u(t, \cdot)\|_{2})\},
\end{align*}
if $2\sqrt{\lambda_{0}+m}<b$, where $L_{j}>0$ ($j=1, 2, 3$) will be specified later.
Now we define the mapping $\Gamma$ on $Z_j$ by
\begin{equation}
\label{MAP-01}
\begin{split}
\Gamma[u](t):= K_{0}(t)\ast_{\H}u_{0} &+ K_{1}(t)\ast_{\H}u_{1}+\int_{0}^{t}K_{1}(t-\tau)\ast_{\H}f(u(\tau))d\tau,
\end{split}
\end{equation}
where $K_0$, $K_1$, and the convolution $\ast_{\H}$ are as defined in Section \ref{SEC:linear}.

We claim that
\begin{equation}
\label{MAP-02}
\|\Gamma[u]\|_{Z_{j}}\leq L_{j}
\end{equation}
for all $u\in Z_{j}$ and
\begin{equation}
\label{MAP-03}
\|\Gamma[u]-\Gamma[v]\|_{Z_{j}}\leq \frac{1}{r_{j}} \|u-v\|_{Z_{j}}
\end{equation}
for all $u, v\in Z_{j}$ with $r_{j}>1$ and $j=1,2,3$. Once we proved \eqref{MAP-02} and \eqref{MAP-03}, we get that $\Gamma$ is a contraction mapping on $Z_{j}$. The Banach fixed point theorem then implies that $\Gamma$ has a unique fixed point on $Z_{j}$ with $j=1,2,3$. It means that there exists a unique global solution $u$ of the equation
$$
u=\Gamma[u] \,\,\, \hbox{in} \,\,\, Z_{j},
$$
which also gives the solution to \eqref{EQ: NoL-01s}.
So, we now concentrate on proving \eqref{MAP-02} and \eqref{MAP-03}.

As we noted before we may identify our Hilbert space $\Sp$ as a measure space $\Sp=L^2(\Omega)$, so that
we can also use the scale of $L^p$ spaces on $\Omega$. We can write $\|\cdot\|_{\Sp}=\|\cdot\|_2$ in this notation.
Recalling the second assumption in \eqref{PR: f-007} on $f$, namely,
$$
|f(u)-f(v)|\leq C (|u|^{p-1}+|v|^{p-1})|u-v|,
$$
applying it to functions $u=u(t)$ and $v=v(t)$ we get
$$
\|(f(u)-f(v))(t, \cdot)\|_{2}^{2}\leq C \int_\Omega (|u(t)|^{p-1}+|v(t)|^{p-1})^{2}|u(t)-v(t)|^{2}.
$$
Consequently, by the H\"{o}lder inequality, we get
$$
\|(f(u)-f(v))(t, \cdot)\|_{2}^{2}\leq C (\|u(t, \cdot)\|^{p-1}_{2p}+\|v(t, \cdot)\|^{p-1}_{2p})^{2} \|(u-v)(t, \cdot)\|^{2}_{2p}
$$
since
$$
\frac{1}{\frac{p}{p-1}}+\frac{1}{p}=1.
$$
By the Gagliardo--Nirenberg-type inequality \eqref{Inty-G-N-01} which holds for $p$ by the assumption, and by Young's inequality
$$
a^{\theta} b^{1-\theta}\leq \theta a + (1-\theta) b
$$
for $0\leq\theta\leq1$, $a,b\geq0$,
we obtain
\begin{equation}
\label{EQ: Gagliardo-Nirenberg-01}
\begin{split}
\|(f(u)& -f(v))(t, \cdot)\|_{2} \leq C \Big[\left(\|\H^{1/2} u(t, \cdot)\|_{2}+\|u(t, \cdot)\|_{2}\right)^{p-1}\\
& +\left(\|\H^{1/2} v(t, \cdot)\|_{2}+\|v(t, \cdot)\|_{2}\right)^{p-1}\Big]
\\
& \times \left(\|\H^{1/2} (u-v)(t, \cdot)\|_{2}+\|(u-v)(t, \cdot)\|_{2}\right).
\end{split}
\end{equation}
Recalling that $\|u\|_{Z_{j}}\leq L_{j}$ and $\|v\|_{Z_{j}}\leq L_{j}$ for $j=1,2,3$, from \eqref{EQ: Gagliardo-Nirenberg-01} we get
\begin{equation}
\label{EQ: Gagliardo-Nirenberg-02}
\|(f(u)-f(v))(t, \cdot)\|_{2}  \leq C (1+t)^{p/2}\e[-\frac{b}{2} p t] L_{1}^{p-1} \|u-v\|_{Z_{1}},
\end{equation}
for $0<b< 2\sqrt{\lambda_{0}+m}$, and
\begin{equation}
\label{EQ: Gagliardo-Nirenberg-02a}
\|(f(u)-f(v))(t, \cdot)\|_{2}  \leq C (1+t)^{3p/2}\e[-\frac{b}{2} p t] L_{2}^{p-1} \|u-v\|_{Z_{2}},
\end{equation}
for $b = 2\sqrt{\lambda_{0}+m}$, and
\begin{equation}
\label{EQ: Gagliardo-Nirenberg-02b}
\|(f(u)-f(v))(t, \cdot)\|_{2}  \leq C (1+t)^{p/2} \e[-(\frac{b}{2}-\sqrt{\frac{b^{2}}{4}-\lambda_{0}-m}) p t] L_{3}^{p-1} \|u-v\|_{Z_{3}},
\end{equation}
for $2\sqrt{\lambda_{0}+m}<b$.

By putting $v=0$ in \eqref{EQ: Gagliardo-Nirenberg-02}--\eqref{EQ: Gagliardo-Nirenberg-02b}, and using that $f(0)=0$, we also have
\begin{equation}
\label{EQ: Gagliardo-Nirenberg-03}
\begin{split}
\|f(u)(t, \cdot)\|_{2} & \leq C (1+t)^{p/2}\e[-\frac{b}{2} p t] L_{1}^{p},\quad \textrm{ for } \; 0 < b < 2\sqrt{\lambda_{0}+m},
\end{split}
\end{equation}
and
\begin{equation}
\label{EQ: Gagliardo-Nirenberg-03a}
\begin{split}
\|f(u)(t, \cdot)\|_{2} & \leq C (1+t)^{3p/2}\e[-\frac{b}{2} p t] L_{2}^{p},\quad \textrm{ for } \; b = 2\sqrt{\lambda_{0}+m},
\end{split}
\end{equation}
and
\begin{equation}
\label{EQ: Gagliardo-Nirenberg-03b}
\begin{split}
\|f(u)(t, \cdot)\|_{2} & \leq C (1+t)^{p/2}\e[-(\frac{b}{2}-\sqrt{\frac{b^{2}}{4}-\lambda_{0}-m}) p t] L_{3}^{p},
\quad \textrm{ for } \; 2\sqrt{\lambda_{0}+m}<b.
\end{split}
\end{equation}

Now, let us estimate the integral operator
\begin{equation}
\label{OP: Int-NoL-01}
\begin{split}
J[u](t,x):=\int_{0}^{t}K_{1}(t-\tau)\ast_{\H}f(u(\tau,x))d\tau.
\end{split}
\end{equation}
More precisely, for $\alpha=0, 1$ and for all $\beta\geq 0$ we have
\begin{equation*}
\label{OP: Int-NoL-02}
\begin{split}
|\partial^{\alpha}_{t} & \H^{\beta}J[u](t, x)|^{2}\leq \Big| \int_{0}^{t}\partial^{\alpha}_{t}\H^{\beta}K_{1}(t-\tau)\ast_{\H}f(u(\tau, x)) d \tau \Big|^{2} \\
&\leq \left(\int_{0}^{t} \Big| \partial^{\alpha}_{t}\H^{\beta}K_{1}(t-\tau)\ast_{\H}f(u(\tau, x)) \Big| d \tau \right)^{2} \\
&\leq t \int_{0}^{t} \Big| \partial^{\alpha}_{t}\H^{\beta}K_{1}(t-\tau)\ast_{\H}f(u(\tau, x)) \Big|^{2} d \tau.
\end{split}
\end{equation*}
Then for $0<b< 2\sqrt{\lambda_{0}+m}$, by using Proposition \ref{LEM: Est-01}, we get
\begin{equation}
\label{OP: Int-NoL-03}
\begin{split}
&\|\partial^{\alpha}_{t}\H^{\beta} J[u](t, \cdot)\|_{2}^{2} \leq t \int_{0}^{t}\| \partial^{\alpha}_{t}\H^{\beta} K_{1}(t-\tau)\ast_{\H}f(u(\tau, \cdot))  \|_{2}^{2} d \tau \\
&\leq C t \int_{0}^{t} \e[-2\frac{b}{2}(t-\tau)] \|f(u(\tau, \cdot)) \|_{H^{\alpha-1+2\beta}_{\H}}^{2} d \tau \\
& = C t \e[ - b t ] \int_{0}^{t} \e[ b \tau ] \|f(u(\tau, \cdot)) \|_{H^{\alpha-1+2\beta}_{\H}}^{2} d \tau.
\end{split}
\end{equation}
Similarly, for $b = 2\sqrt{\lambda_{0}+m}$ we obtain
\begin{equation}
\label{OP: Int-NoL-03a}
\begin{split}
&\|\partial^{\alpha}_{t}\H^{\beta} J[u](t, \cdot)\|_{2}^{2} \leq C t \e[ - b t ] \int_{0}^{t} (1+t-\tau)^{2} \e[ b \tau ] \|f(u(\tau, \cdot)) \|_{H^{\alpha-1+2\beta}_{\H}}^{2} d \tau \\
&\leq C t(1+t)^{2} \e[ - b t ] \int_{0}^{t} \e[ b \tau ] \|f(u(\tau, \cdot)) \|_{H^{\alpha-1+2\beta}_{\H}}^{2} d \tau.
\end{split}
\end{equation}
Also, for $2\sqrt{\lambda_{0}+m}<b$ we have
\begin{equation}
\label{OP: Int-NoL-03b}
\begin{split}
\|\partial^{\alpha}_{t}\H^{\beta} J[u](t, \cdot)&\|_{2}^{2} \leq C t \e[- 2(\frac{b}{2}-\sqrt{\frac{b^{2}}{4}-\lambda_{0}-m}) t] \\
&\times\int_{0}^{t} \e[  2(\frac{b}{2}-\sqrt{\frac{b^{2}}{4}-\lambda_{0}-m}) \tau] \|f(u(\tau, \cdot)) \|_{H^{\alpha-1+2\beta}_{\H}}^{2} d \tau.
\end{split}
\end{equation}
Now we have to control the norm
$
\|f(u(\tau, \cdot)) \|_{H^{\alpha-1+2\beta}_{\H}}^{2}.
$
We notice that for $(\alpha, \beta)=(0, 1/2)$ and $(\alpha, \beta)=(1, 0)$ we have $\alpha-1+2\beta\leq0$.

Thus, using \eqref{EQ: Gagliardo-Nirenberg-02} and \eqref{EQ: Gagliardo-Nirenberg-03}, %for $p>1$
we obtain from \eqref{OP: Int-NoL-03} that
\begin{equation}
\label{OP: Int-NoL-04}
\|\partial^{\alpha}_{t} \H^{\beta} ( J[u] - J[v] )(t, \cdot)\|_{2} \leq C t^{1/2} \e[- \frac{b}{2} t] \, L_{1}^{p-1}\|u-v\|_{Z_{1}},
\end{equation}
and
\begin{equation}
\label{OP: Int-NoL-05}
\|\partial^{\alpha}_{t} \H^{\beta} J[u](t, \cdot)\|_{2} \leq C t^{1/2} \e[- \frac{b}{2} t] \, L^{p}_{1},
\end{equation}
%For $p=1$ we have
%\begin{equation}
%\label{OP: Int-NoL-04-c}
%\|\partial^{\alpha}_{t} \H^{\beta} ( J[u] - J[v] )(t, \cdot)\|_{2} \leq C t \e[- \frac{b}{2} t] \, \|u-v\|_{Z_{1}},
%\end{equation}
%and
%\begin{equation}
%\label{OP: Int-NoL-05-c}
%\|\partial^{\alpha}_{t} \H^{\beta} J[u](t, \cdot)\|_{2} \leq C t \e[- \frac{b}{2} t] \, L_{1},
%\end{equation}
%with all the estimates \eqref{OP: Int-NoL-04}--\eqref{OP: Int-NoL-05-c} holding
with the estimates \eqref{OP: Int-NoL-04}--\eqref{OP: Int-NoL-05} holding
for $(\alpha, \beta)=(0, 1/2)$ and $(\alpha, \beta)=(1, 0)$.

Similarly, using \eqref{EQ: Gagliardo-Nirenberg-02a} and \eqref{EQ: Gagliardo-Nirenberg-03a}, %for $p>1$
we get from \eqref{OP: Int-NoL-03a}:
\begin{equation}
\label{OP: Int-NoL-04a}
\|\partial^{\alpha}_{t} \H^{\beta} ( J[u] - J[v] )(t, \cdot)\|_{2} \leq C t^{1/2}(1+t) \e[- \frac{b}{2} t] \, L_{2}^{p-1}\|u-v\|_{Z_{2}},
\end{equation}
and
\begin{equation}
\label{OP: Int-NoL-05a}
\|\partial^{\alpha}_{t} \H^{\beta} J[u](t, \cdot)\|_{2} \leq C t^{1/2}(1+t) \e[- \frac{b}{2} t] \, L^{p}_{2},
\end{equation}
with the estimates \eqref{OP: Int-NoL-04}--\eqref{OP: Int-NoL-05} holding
for $(\alpha, \beta)=(0, 1/2)$ and $(\alpha, \beta)=(1, 0)$.

Also, from \eqref{OP: Int-NoL-03b} by using \eqref{EQ: Gagliardo-Nirenberg-02b} and \eqref{EQ: Gagliardo-Nirenberg-03b}, %for $p>1$
we get
\begin{equation}
\label{OP: Int-NoL-04b}
\|\partial^{\alpha}_{t} \H^{\beta} ( J[u] - J[v] )(t, \cdot)\|_{2} \leq C t^{1/2} \e[- (\frac{b}{2}-\sqrt{\frac{b^{2}}{4}-\lambda_{0}-m}) t] \, L_{3}^{p-1}\|u-v\|_{Z_{3}},
\end{equation}
\begin{equation}
\label{OP: Int-NoL-05b}
\|\partial^{\alpha}_{t} \H^{\beta} J[u](t, \cdot)\|_{2} \leq C t^{1/2} \e[- (\frac{b}{2}-\sqrt{\frac{b^{2}}{4}-\lambda_{0}-m}) t] \, L_{3}^{p},
\end{equation}
%and for $p=1$ we get
%\begin{equation*}
%\label{OP: Int-NoL-04b-c}
%\|\partial^{\alpha}_{t} \H^{\beta} ( J[u] - J[v] )(t, \cdot)\|_{2} \leq C t  \e[- (\frac{b}{2}-\sqrt{\frac{b^{2}}{4}-\lambda_{1}}) t] \, L_{2}^{p-1}\|u-v\|_{Z_{2}},
%\end{equation*}
%\begin{equation*}
%\label{OP: Int-NoL-05b-c}
%\|\partial^{\alpha}_{t} \H^{\beta} J[u](t, \cdot)\|_{2} \leq C t \e[- (\frac{b}{2}-\sqrt{\frac{b^{2}}{4}-\lambda_{1}}) t] \, L_{2}^{p},
%\end{equation*}
for $(\alpha, \beta)=(0, 1/2)$ and $(\alpha, \beta)=(1, 0)$.

Consequently, by the definition of $\Gamma[u]$ in \eqref{MAP-01} and
using Proposition \ref{LEM: Est-01} for the first term and estimates for $\|J[u]\|_{Z_{j}}$ for the second term below, we obtain
%\footnote{\bf NIYAZ: we still need an additional argument to overcome the factor $t^{1/2}$.}
\begin{equation}
\label{Gamma: Contraction mapping-01}
\begin{split}
\|\Gamma[u]\|_{Z_{j}} & \leq \|K_{0}(t)\ast_{\H}u_{0} + K_{1}(t)\ast_{\H}u_{1}\|_{Z_{j}} + \|J[u]\|_{Z_{j}} \\
& \leq C_{1j}(\|u_{0}\|_{H^{1}_{\H}}+\|u_{1}\|_{L^{2}}) + C_{2j}L_{j}^{p},
\end{split}
\end{equation}
for some $C_{1j}>0$ and $C_{2j}>0$, $j=1,2,3$.

Moreover, in the similar way, we can estimate
\begin{equation}
\label{Gamma: Contraction mapping-02}
\|\Gamma[u]-\Gamma[v]\|_{Z_{j}} \leq \|J[u] - J[v]\|_{Z_{j}} \leq C_{3j}L_j^{p-1} \|u-v\|_{Z_{j}},
\end{equation}
for some $C_{3j}>0$, $j=1,2,3$. Taking some $r_{j}>1$, we choose $$L_{j}:=r_{j} C_{1j}(\|u_{0}\|_{H^{1}_{\H}}+\|u_{1}\|_{L^{2}})$$ with sufficiently small $\|u_{0}\|_{H^{1}_{\H}}+\|u_{1}\|_{L^{2}}<\varepsilon$ so that
\begin{equation}
\label{Gamma: Contraction mapping-03}
C_{2j}L_{j}^{p}\leq \frac{1}{r_{j}} L_{j}, \,\,\,\, C_{3j}L_{j}^{p-1}\leq \frac{1}{r_{j}}.
\end{equation}
Then estimates \eqref{Gamma: Contraction mapping-01}--\eqref{Gamma: Contraction mapping-03} imply the desired estimates \eqref{MAP-02} and \eqref{MAP-03}. This means that we can apply the fixed point theorem for the existence of solutions.

The estimates \eqref{TH-NoL-1-01} and \eqref{TH-NoL-1-03} follow from \eqref{OP: Int-NoL-03}-- \eqref{OP: Int-NoL-03b}. Theorem \ref{TH: 01} is now proved.
\end{proof}

\section{Nonlinear damped wave equation}
\label{S: n+2}

In this section we deal with the general nonlinearity by considering
the nonlinear term of the form $F(u, u_t, \H^{1/2} u)$, for some function $F:\mathbb C^3\to\mathbb C$. Now, let us suppose that the following property holds.

Denoting $$U:=(u, u_t, \H^{1/2} u)$$ for $u\in C^{1}(\mathbb R_{+}; H_{\H}^{1}),$
we assume that $F(U)\in C(\mathbb R_{+}; H_{\H}^{1})$ and we call the index $p>1$ to be $(F,\H)$-admissible if we have the estimate
\begin{equation} \label{PR: f-02}
\begin{split}
\|F(U)-F(V)\|_{H_{\H}^{1}}\lesssim (\|U\|_{H_{\H}^{1}}^{p-1}+ \|V\|_{H_{\H}^{1}}^{p-1})\|U-V\|_{H_{\H}^{1}}.
\end{split}
\end{equation}

We note that using the definition of Sobolev spaces in \eqref{EQ:HsL}, in this notation we have
\begin{equation}\label{EQ:Uu}
\|U\|_{H_{\H}^{1}}\simeq \|\H^{1/2}u\|_{\Sp}+\|\H^{1/2}\partial_t u\|_{\Sp}+\|\H u\|_{\Sp}.
\end{equation}

An example of $F$ satisfying \eqref{PR: f-02} may be given by nonlinearities of the form
\begin{equation}\label{EQ:exnon}
F(U)=\varphi \|U\|_{\Sp}^p \quad \textrm{ or }\quad F(U)=\varphi \|U\|_{H_{\H}^{1}}^p,
\end{equation}
for some $\varphi\in {\rm Dom}\,(\H)$.

We now give the global in time well-posedness statement.

\begin{thm} \label{TH: 02}
Let $p>1$ be $(F,\H)$-admissible, i.e. assume that $F=F(u, \partial_{t}u, \H^{1/2} u)$ satisfies the condition \eqref{PR: f-02}.
Suppose that $F(0)=0$, and that $u_{0}\in H_{\H}^{2}$ and $u_{1}\in H_{\H}^{1}$ are such that
\begin{equation*} \label{EQ: Th-cond-01}
\|u_{0}\|_{H_{\H}^{2}}+\|u_{1}\|_{H_{\H}^{1}}\leq\varepsilon.
\end{equation*}
Assume that $\lambda_0\geq 0$ and $\lambda_0+m>0$.
Then, there exists a small positive constant $\varepsilon_{0}>0$ such that the Cauchy problem \begin{equation*}\label{EQ: NoL-02}
\left\{ \begin{split}
\partial_{t}^{2}u(t)+\H u(t)+b\partial_{t}u(t)+m u(t)&=F(u, \partial_{t}u, \H^{1/2} u), \quad t>0,\\
u(0)&=u_{0}\in H_{\H}^{2}, \\
\partial_{t}u(0)&=u_{1}\in H_{\H}^{1},
\end{split}
\right.
\end{equation*}
has a unique global solution $u\in C(\mathbb R_{+}; H_{\H}^{2})\cap C^{1}(\mathbb R_{+}; H_{\H}^{1})$ for all $0<\varepsilon\leq\varepsilon_{0}$.

Moreover, for $0 < b < 2\sqrt{\lambda_{0}+m}$ we have
\begin{equation}\label{TH-NoL-1-01b}
\|\partial_{t}^{\alpha}\H^{\beta}u(t)\|_{\Sp}\lesssim (1+t)^{1/2}\e[-\frac{b}{2}t],
\end{equation}
and for $b = 2\sqrt{\lambda_{0}+m}$ we have
\begin{equation}\label{TH-NoL-1-02b}
\|\partial_{t}^{\alpha}\H^{\beta}u(t)\|_{\Sp}\lesssim (1+t)^{3/2}\e[-\frac{b}{2}t],
\end{equation}
and for $2\sqrt{\lambda_{0}+m}<b$ we have
\begin{equation}\label{TH-NoL-1-03b} \|\partial_{t}^{\alpha}\H^{\beta}u(t)\|_{\Sp}\lesssim (1+t)^{1/2} \e[-(\frac{b}{2}-\sqrt{\frac{b^{2}}{4}-\lambda_{0}-m})t],
\end{equation}
for any $(\alpha, \beta)\in\{(0, 0), (0, 1/2), (1, 0), (0, 1), (1, 1/2), (2, 0)\}$.
\end{thm}

\begin{proof} The proof of Theorem \ref{TH: 02} is similar to that of Theorem \ref{TH: 01} except that we aim at using the assumption \eqref{PR: f-02} instead of the Gagliardo-Nirenberg inequality.
First, we define the closed subsets $Z_j$ of the space $C^{2}(\mathbb R_{+}; \,\, H^{2}_{\H})$ by
$$
Z_{j} :=\{u\in C^{2}(\mathbb R_{+}; \,\, H^{2}_{\H}); \,\, \|u\|_{Z_{j}}\leq L_{j}\}, \,\,\, j=4,5,6,
$$
with
\begin{equation*}
\begin{split}
\|u\|_{Z_{4}}:=\sup_{t\geq0}\{&(1+t)^{-1/2}\e[\frac{b}{2}t](\|u(t, \cdot)\|_{2}+\|\partial_{t}u(t, \cdot)\|_{2}+\|\H^{1/2}u(t, \cdot)\|_{2}\\
&+\|\partial_{t}\H^{1/2}u(t, \cdot)\|_{2}+\|\H u(t, \cdot)\|_{2}+\|\partial_{t}^{2}u(t, \cdot)\|_{2})\}, \,\,\, \hbox{if} \,\,\, 0 < b < 2\sqrt{\lambda_{0}+m},
\end{split}
\end{equation*}
and
\begin{equation*}
\begin{split}
\|u\|_{Z_{5}}:=\sup_{t\geq0}\{&(1+t)^{-3/2}\e[\frac{b}{2}t](\|u(t, \cdot)\|_{2}+\|\partial_{t}u(t, \cdot)\|_{2}+\|\H^{1/2}u(t, \cdot)\|_{2}\\
&+\|\partial_{t}\H^{1/2}u(t, \cdot)\|_{2}+\|\H u(t, \cdot)\|_{2}+\|\partial_{t}^{2}u(t, \cdot)\|_{2})\}, \,\,\, \hbox{if} \,\,\, b = 2\sqrt{\lambda_{0}+m},
\end{split}
\end{equation*}
and
\begin{equation*}
\begin{split}
\|u\|_{Z_{6}}:=\sup_{t\geq0}\{&(1+t)^{-1/2}\e[(\frac{b}{2}-\sqrt{\frac{b^{2}}{4}-\lambda_{0}-m})t](\|u(t, \cdot)\|_{2}+\|\partial_{t}u(t, \cdot)\|_{2}+\|\H^{1/2}u(t, \cdot)\|_{2}\\
&+\|\partial_{t}\H^{1/2}u(t, \cdot)\|_{2}+\|\H u(t, \cdot)\|_{2}+\|\partial_{t}^{2}u(t, \cdot)\|_{2})\}, \,\,\, \hbox{if} \,\,\, 2\sqrt{\lambda_{0}+m}<b,
\end{split}
\end{equation*}
where $L_{j}>0$ ($j=4,5,6$) are to be specified later.

Now, we begin by repeating several steps  from the proof of Theorem \ref{TH: 01}, namely, we define the mapping $\Gamma$ on $Z_4$, $Z_5$ and $Z_6$ by
\begin{equation}
\label{MAP-01b}
\begin{split}
\Gamma[u](t):= K_{0}(t)\ast_{\H}u_{0} &+ K_{1}(t)\ast_{\H}u_{1}\\
&+\int_{0}^{t}K_{1}(t-\tau)\ast_{\H}F(u, u_{t}, \H^{1/2}u)(\tau)d\tau,
\end{split}
\end{equation}
and we show that $\Gamma$ is a contraction mapping on $Z_4$, $Z_5$ and $Z_6$. By \eqref{PR: f-02} we have
\begin{equation} \label{PR: f-02b}
\begin{split}
\|F(U)-F(V)\|_{H^{1}_{\H}}\lesssim (\|U\|_{H^{1}_{\H}}^{p-1}+\|V\|_{H^{1}_{\H}}^{p-1})\|U-V\|_{H^{1}_{\H}}.
\end{split}
\end{equation}
We take $u$ and $v$ satisfying $\|u\|_{Z_{j}}\leq L_{j}$ and $\|v\|_{Z_{j}}\leq L_{j}$ for $j=4,5,6$.
Recalling \eqref{EQ:Uu} for $U=(u,\partial_t u,\H^{1/2}u)$,
from \eqref{PR: f-02b} we get
\begin{equation}
\label{EQ: Gagliardo-Nirenberg-02-cc}
\|(F(U)-F(V))(t, \cdot)\|_{H^{1}_{\H}}  \leq C (1+t)^{p/2} \e[-\frac{b}{2} p t] L_{4}^{p-1} \|u-v\|_{Z_{4}},
\end{equation}
and
\begin{equation}
\label{EQ: Gagliardo-Nirenberg-02a-cc}
\|(F(U)-F(V))(t, \cdot)\|_{H^{1}_{\H}}  \leq C (1+t)^{3p/2} \e[-\frac{b}{2} p t] L_{5}^{p-1} \|u-v\|_{Z_{5}},
\end{equation}
and
\begin{equation}
\label{EQ: Gagliardo-Nirenberg-02b-cc}
\|(F(U)-F(V))(t, \cdot)\|_{H^{1}_{\H}}  \leq C (1+t)^{p/2} \e[-(\frac{b}{2}-\sqrt{\frac{b^{2}}{4}-\lambda_{0}-m}) p t] L_{6}^{p-1} \|u-v\|_{Z_{6}},
\end{equation}
respectively.
Since $F(0)=0$, by putting $v=0$ and $V=0$ in \eqref{EQ: Gagliardo-Nirenberg-02-cc}--\eqref{EQ: Gagliardo-Nirenberg-02b-cc}, we obtain
\begin{equation}
\label{EQ: Gagliardo-Nirenberg-03-cc}
\|F(U)(t, \cdot)\|_{H^{1}_{\H}}  \leq C (1+t)^{p/2} \e[-\frac{b}{2} p t] L_{4}^{p},
\end{equation}
and
\begin{equation}
\label{EQ: Gagliardo-Nirenberg-03a-cc}
\|F(U)(t, \cdot)\|_{H^{1}_{\H}}  \leq C (1+t)^{3p/2} \e[-\frac{b}{2} p t] L_{5}^{p},
\end{equation}
and
\begin{equation}
\label{EQ: Gagliardo-Nirenberg-03b-cc}
\|F(U)(t, \cdot)\|_{H^{1}_{\H}}  \leq C (1+t)^{p/2} \e[-(\frac{b}{2}-\sqrt{\frac{b^{2}}{4}-\lambda_{0}-m}) p t] L_{6}^{p},
\end{equation}
respectively.
As in the proof of Theorem \ref{TH: 01}, in view of Proposition \ref{LEM: Est-01}, for the integral operator
\begin{equation}
\label{OP: Int-NoL-01-cc}
\begin{split}
J[u](t, x):=\int_{0}^{t}K_{1}(t-\tau)\ast_{\H}F(u(\tau, x), u_t(\tau, x), \H^{1/2} u(\tau, x))d\tau,
\end{split}
\end{equation}
for $0<b< 2\sqrt{\lambda_{0}+m}$ we have
\begin{equation}
\label{OP: Int-NoL-03-cc}
\begin{split}
&\|\partial^{\alpha}_{t}\H^{\beta} J[u](t, \cdot)\|_{2}^{2} \leq t \int_{0}^{t}\| \partial^{\alpha}_{t}\H^{\beta} K_{1}(t-\tau)\ast_{\H}F(u, u_t, \H^{1/2} u)(\tau, \cdot)  \|_{2}^{2} d \tau \\
&\leq C t \e[- b t] \int_{0}^{t} \e[  b \tau] \|F(u, u_t, \H^{1/2} u)(\tau, \cdot) \|_{H^{\alpha-1+2\beta}_{\H}}^{2} d \tau.
\end{split}
\end{equation}
Similarly, for $b=2\sqrt{\lambda_{0}+m}$ we get
\begin{equation}
\label{OP: Int-NoL-03a-cc}
\begin{split}
\|\partial^{\alpha}_{t}\H^{\beta} J[u](t, \cdot)\|_{2}^{2} \leq  C t (1+t)^{2} \e[- b t] \int_{0}^{t} \e[  b \tau] \|F(u, u_t, \H^{1/2} u)(\tau, \cdot) \|_{H^{\alpha-1+2\beta}_{\H}}^{2} d \tau.
\end{split}
\end{equation}
Also, for $2\sqrt{\lambda_{0}+m}<b$ we obtain
\begin{equation}
\label{OP: Int-NoL-03b-cc}
\begin{split}
\|\partial^{\alpha}_{t}\H^{\beta} J[u](t, \cdot)&\|_{2}^{2} \leq C t \e[- 2(\frac{b}{2}-\sqrt{\frac{b^{2}}{4}-\lambda_{0}-m}) t] \\
&\times\int_{0}^{t} \e[  2(\frac{b}{2}-\sqrt{\frac{b^{2}}{4}-\lambda_{0}-m}) \tau] \|F(u, u_t, \H^{1/2} u)(\tau, \cdot) \|_{H^{\alpha-1+2\beta}_{\H}}^{2} d \tau
\end{split}
\end{equation}
for $(\alpha, \beta)\in\{(0, 0), (0, 1/2), (1, 0), (0, 1), (1, 1/2), (2, 0)\}$.

Now, combining the previous discussions with the estimates \eqref{EQ: Gagliardo-Nirenberg-02-cc}--\eqref{EQ: Gagliardo-Nirenberg-03b-cc}, we complete the proof of  Theorem \ref{TH: 02}.
\end{proof}

\section{Higher order nonlinearities}
\label{SEC:higher}

In this section
 we briefly indicate how the obtained results can be extended to higher order nonlinearities $F_l:{\mathbb C}^{2l+1}\to {\mathbb C}$, $l\in\mathbb N$, in the form
\begin{equation}\label{EQ:F-gen}
F_l=F_l(u,\{\partial^j_t u\}_{j=1}^l, \{{\mathcal L}^{j/2} u\}_{j=1}^{l}).
\end{equation}
Denoting $$U:=(u,\{\partial^j_t u\}_{j=1}^l, \{{\mathcal L}^{j/2} u\}_{j=1}^{l})$$ for $u\in C^{l}(\mathbb R_{+}; H_{\H}^{l})$, we assume that $F_l(U)\in C(\mathbb R_{+}; H_{\H}^{l})$ and we say that the index $p>1$ is $(F_l,\H)$-admissible if we have the inequality
\begin{equation} \label{PR: f-02m}
\begin{split}
\|F_l(U)-F_l(V)\|_{H_{\H}^{l}}\lesssim (\|U\|_{H_{\H}^{1}}^{p-1}+ \|V\|_{H_{\H}^{1}}^{p-1})\|U-V\|_{H_{\H}^{1}}.
\end{split}
\end{equation}
An example of $F=F_l$ satisfying \eqref{PR: f-02m} may be given by nonlinearities of the form
\begin{equation}\label{EQ:exnonh}
F(U)=\varphi \|U\|_{\Sp}^p \quad \textrm{ or }\quad F(U)=\varphi \|U\|_{H_{\H}^{1}}^p,
\end{equation}
for some $\varphi\in H_\H^{1}$.

\begin{thm} \label{TH: 0g}
Let $l\in\mathbb N$. Assume that $\lambda_0\geq0$ and $\lambda_0+m>0$, where $\lambda_0=\lambda_0(\H)$ is the bottom of the spectrum of $\H$. Let $p>1$ be such that $F_l$ as in \eqref{EQ:F-gen} satisfies the condition \eqref{PR: f-02m}.
Suppose that $F(0)=0$, and that $u_{0}\in H_{\H}^{l+1}$ and $u_{1}\in H_{\H}^{l}$ are such that
\begin{equation*} \label{EQ: Th-cond-01m}
\|u_{0}\|_{H_{\H}^{l+1}}+\|u_{1}\|_{H_{\H}^{l}}\leq\varepsilon.
\end{equation*}
Then, there exists a small positive constant $\varepsilon_{0}>0$ such that the Cauchy problem \begin{equation*}\label{EQ: NoL-02m}
\left\{ \begin{split}
\partial_{t}^{2}u(t)+\H u(t)+\partial_{t}u(t)+m u(t)&=F_l(u,\{\partial^j u\}_{j=1}^l, \{{\mathcal L}^{j/2} u\}_{j=1}^{l}), \quad t>0,\\
u(0)&=u_{0}\in H_{\H}^{l+1}, \\
\partial_{t}u(0)&=u_{1}\in H_{\H}^{l},
\end{split}
\right.
\end{equation*}
has a unique global solution $u\in C(\mathbb R_{+}; H_{\H}^{l+1})\cap C^{1}(\mathbb R_{+}; H_{\H}^{l})$ for all $0<\varepsilon\leq\varepsilon_{0}$.

Moreover, for $0<b< 2\sqrt{\lambda_{0}+m}$ we have
\begin{equation}\label{TH-NoL-1-01bm}
\|\partial_{t}^{\alpha}\H^{\beta}u(t)\|_{\Sp}\lesssim (1+t)^{1/2}\e[-\frac{b}{2}t],
\end{equation}
for $b = 2\sqrt{\lambda_{0}+m}$ we have
\begin{equation}\label{TH-NoL-1-02bm}
\|\partial_{t}^{\alpha}\H^{\beta}u(t)\|_{\Sp}\lesssim (1+t)^{3/2}\e[-\frac{b}{2}t],
\end{equation}
and for $2\sqrt{\lambda_{0}+m}<b$ we have
\begin{equation}\label{TH-NoL-1-03bm}
\|\partial_{t}^{\alpha}\H^{\beta}u(t)\|_{\Sp}\lesssim (1+t)^{1/2} \e[-(\frac{b}{2}-\sqrt{\frac{b^{2}}{4}-\lambda_{0}-m})t],
\end{equation}
for any $(\alpha, \beta)\in \mathbb N_0\times\frac12\mathbb N_0$ with $\alpha+2\beta\leq l+1$.
%\footnote{check this}
\end{thm}

\begin{proof} The proof of Theorem \ref{TH: 0g} is similar to that of Theorem \ref{TH: 02} except that we aim at using more general assumption \eqref{PR: f-02m} instead of the assumption \eqref{PR: f-02}.
Analogously, define the closed subsets $Z_j$ of the space $C^{l+1}(\mathbb R_{+}; \,\, H^{l+1}_{\H})$ by
$$
Z_{j} :=\{u\in C^{l+1}(\mathbb R_{+}; \,\, H^{l+1}_{\H}); \,\, \|u\|_{Z_{j}}\leq L_{j}\}, \,\,\, j=3l+1,3l+2,3l+3
$$
with
\begin{equation*}
\begin{split}
\|u\|_{Z_{3l+1}}:=\sup_{t\geq0}\{(1+t)^{-1/2}\e[\frac{b}{2}t]\left(\sum_{(\alpha, \beta)\in \mathbb N_0\times\frac12\mathbb N_0}^{\alpha+2\beta\leq l+1}\|\partial_{t}^{\alpha}\H^{\beta}u(t, \cdot)\|_{2}\right)\},
\end{split}
\end{equation*}
if $0<b< 2\sqrt{\lambda_{0}+m}$, and
\begin{equation*}
\begin{split}
\|u\|_{Z_{3l+2}}:=\sup_{t\geq0}\{(1+t)^{-3/2}\e[\frac{b}{2}t]\left(\sum_{(\alpha, \beta)\in \mathbb N_0\times\frac12\mathbb N_0}^{\alpha+2\beta\leq l+1}\|\partial_{t}^{\alpha}\H^{\beta}u(t, \cdot)\|_{2}\right)\},
\end{split}
\end{equation*}
if $b = 2\sqrt{\lambda_{0}+m}$, and
\begin{equation*}
\begin{split}
\|u\|_{Z_{3l+3}}:=\sup_{t\geq0}\{(1+t)^{-1/2}\e[(\frac{b}{2}-\sqrt{\frac{b^{2}}{4}-\lambda_{0}-m})t]\left(\sum_{(\alpha, \beta)\in \mathbb N_0\times\frac12\mathbb N_0}^{\alpha+2\beta\leq l+1}\|\partial_{t}^{\alpha}\H^{\beta}u(t, \cdot)\|_{2}\right)\},
\end{split}
\end{equation*}
if $2\sqrt{\lambda_{0}+m}<b$, where $L_{j}>0$ ($j=3l+1, 3l+2, 3l+3$) are to be defined.

As in the proof of Theorem \ref{TH: 02}, we define the mapping $\Gamma_l$ on $Z_{3l+1}$, $Z_{3l+2}$ and $Z_{3l+3}$ by
\begin{equation}
\label{MAP-01b-g}
\begin{split}
\Gamma_l[u](t):= K_{0}(t)\ast_{\H}u_{0} &+ K_{1}(t)\ast_{\H}u_{1}\\
&+\int_{0}^{t}K_{1}(t-\tau)\ast_{\H}F_l(u,\{\partial^j u\}_{j=1}^l, \{{\mathcal L}^{j/2} u\}_{j=1}^{l})(\tau)d\tau.
\end{split}
\end{equation}
In addition, we establish that $\Gamma_l$ is a contraction mapping on $Z_{3l+1}$, $Z_{3l+2}$ and $Z_{3l+3}$. By \eqref{PR: f-02m} we have
\begin{equation} \label{PR: f-02bg}
\begin{split}
\|F_l(U)-F_l(V)\|_{H_{\H}^{l}}\lesssim (\|U\|_{H_{\H}^{1}}^{p-1}+ \|V\|_{H_{\H}^{1}}^{p-1})\|U-V\|_{H_{\H}^{1}}.
\end{split}
\end{equation}

Take $u$ and $v$ such that $\|u\|_{Z_{j}}\leq L_{j}$ and $\|v\|_{Z_{j}}\leq L_{j}$ for $j=3l+1, 3l+2, 3l+3$.
Then for $U=(u,\{\partial^j u\}_{j=1}^l, \{{\mathcal L}^{j/2} u\}_{j=1}^{l})$,
from \eqref{PR: f-02bg} we have
\begin{equation}
\label{EQ: Gagliardo-Nirenberg-02-ccg}
\|(F(U)-F(V))(t, \cdot)\|_{H^{l}_{\H}}  \leq C (1+t)^{p/2} \e[-\frac{b}{2} p t] L_{3l+1}^{p-1} \|u-v\|_{Z_{3l+1}},
\end{equation}
and
\begin{equation}
\label{EQ: Gagliardo-Nirenberg-02a-ccg}
\|(F(U)-F(V))(t, \cdot)\|_{H^{l}_{\H}}  \leq C (1+t)^{3p/2} \e[-\frac{b}{2} p t] L_{3l+2}^{p-1} \|u-v\|_{Z_{3l+2}},
\end{equation}
and
\begin{equation}
\label{EQ: Gagliardo-Nirenberg-02b-ccg}
\|(F(U)-F(V))(t, \cdot)\|_{H^{l}_{\H}}  \leq C (1+t)^{p/2} \e[-(\frac{b}{2}-\sqrt{\frac{b^{2}}{4}-\lambda_{0}-m}) p t] L_{3l+3}^{p-1} \|u-v\|_{Z_{3l+3}},
\end{equation}
respectively.
Due to $F(0)=0$, by substituting $V=0$ in \eqref{EQ: Gagliardo-Nirenberg-02-ccg}--\eqref{EQ: Gagliardo-Nirenberg-02b-ccg}, we get
\begin{equation}
\label{EQ: Gagliardo-Nirenberg-03-ccg}
\|F(U)(t, \cdot)\|_{H^{l}_{\H}}  \leq C (1+t)^{p/2} \e[-\frac{b}{2} p t] L_{3l+1}^{p},
\end{equation}
and
\begin{equation}
\label{EQ: Gagliardo-Nirenberg-03a-ccg}
\|F(U)(t, \cdot)\|_{H^{l}_{\H}}  \leq C (1+t)^{3p/2} \e[-\frac{b}{2} p t] L_{3l+2}^{p},
\end{equation}
and
\begin{equation}
\label{EQ: Gagliardo-Nirenberg-03b-ccg}
\|F(U)(t, \cdot)\|_{H^{l}_{\H}}  \leq C (1+t)^{p/2} \e[-(\frac{b}{2}-\sqrt{\frac{b^{2}}{4}-\lambda_{0}-m}) p t] L_{3l+3}^{p},
\end{equation}
respectively.
Repeating the proof of the above theorem, from the point of view of Proposition \ref{LEM: Est-01}, for the operator
\begin{equation}
\label{OP: Int-NoL-01-ccg}
\begin{split}
J_l[u](t, x):=\int_{0}^{t}K_{1}(t-\tau)\ast_{\H}F_l(u,\{\partial^j u\}_{j=1}^l, \{{\mathcal L}^{j/2} u\}_{j=1}^{l})(\tau, x)d\tau,
\end{split}
\end{equation}
for $0<b< 2\sqrt{\lambda_{0}+m}$ we obtain
\begin{equation}
\label{OP: Int-NoL-03-ccg}
\begin{split}
\|\partial^{\alpha}_{t}&\H^{\beta} J_l[u](t, \cdot)\|_{\Sp}^{2} \\
&\leq t \int_{0}^{t}\| \partial^{\alpha}_{t}\H^{\beta} K_{1}(t-\tau)\ast_{\H}F_l(u,\{\partial^j u\}_{j=1}^l, \{{\mathcal L}^{j/2} u\}_{j=1}^{l})(\tau, \cdot)  \|_{\Sp}^{2} d \tau \\
&\leq C t \e[- b t] \int_{0}^{t} \e[  b \tau] \|F_l(u,\{\partial^j u\}_{j=1}^l, \{{\mathcal L}^{j/2} u\}_{j=1}^{l})(\tau, \cdot) \|_{H^{\alpha-1+2\beta}_{\H}}^{2} d \tau.
\end{split}
\end{equation}
Similarly, for $b=2\sqrt{\lambda_{0}+m}$ we have
\begin{equation}
\label{OP: Int-NoL-03a-ccg}
\begin{split}
\|\partial^{\alpha}_{t}&\H^{\beta} J_l[u](t, \cdot)\|_{2}^{2} \\
  &\leq C t (1+t)^{2} \e[- b t] \int_{0}^{t} \e[  b \tau] \|F_l(u,\{\partial^j u\}_{j=1}^l, \{{\mathcal L}^{j/2} u\}_{j=1}^{l})(\tau, \cdot) \|_{H^{\alpha-1+2\beta}_{\H}}^{2} d \tau.
\end{split}
\end{equation}
Also, for $2\sqrt{\lambda_{0}+m}<b$ we get
\begin{equation}
\label{OP: Int-NoL-03b-ccg}
\begin{split}
\|\partial^{\alpha}_{t}&\H^{\beta} J_l[u](t, \cdot)\|_{2}^{2} \leq C t \e[- 2(\frac{b}{2}-\sqrt{\frac{b^{2}}{4}-\lambda_{0}-m}) t] \\
&\times\int_{0}^{t} \e[  2(\frac{b}{2}-\sqrt{\frac{b^{2}}{4}-\lambda_{0}-m}) \tau] \|F_l(u,\{\partial^j u\}_{j=1}^l, \{{\mathcal L}^{j/2} u\}_{j=1}^{l})(\tau, \cdot) \|_{H^{\alpha-1+2\beta}_{\H}}^{2} d \tau
\end{split}
\end{equation}
for $(\alpha, \beta)\in \mathbb N_0\times\frac12\mathbb N_0$ with $\alpha+2\beta\leq l+1$.

Now, repeating the proof of Theorem \ref{TH: 01} step by step and, taking into account the assumption \eqref{PR: f-02m}, we can finish to prove Theorem \ref{TH: 0g}.
\end{proof}

\appendix
\section{Harmonic oscillator}
\label{APP}

In this appendix we briefly illustrate some constructions of this paper related to the Fourier analysis in the example of the harmonic oscillator \eqref{EQ:ho}, i.e. for
$\Sp=L^2(\Rn)$ and
\begin{equation}\label{EQ:ho-ap}
\H:=-\Delta+|x|^{2}, \,\,\, x\in\mathbb R^{n}.
\end{equation}
In this case we have the bottom of the spectrum $\lambda_{0}=n$.
The operator $\H$ in \eqref{EQ:ho-ap} is essentially self-adjoint on $C_{0}^{\infty}(\mathbb R^{n})$ with eigenvalues
$$
\lambda_{k}=\sum_{j=1}^{n}(2k_{j}+1), \,\,\, k=(k_{1}, \ldots, k_{n})\in\mathbb N_0^{n},
$$
and with eigenfunctions
$$
e_{k}(x)=\prod_{j=1}^{n}P_{k_{j}}(x_{j}){\rm e}^{-\frac{|x|^{2}}{2}},
$$
which form an orthogonal system in $\Sp=L^{2}(\mathbb R^{n})$. Here, $P_{m}(\cdot)$ is the $m$--th order Hermite polynomial, and
$$
P_{m}(t)=c_{m}{\rm e}^{\frac{|t|^{2}}{2}}\left(t-\frac{d}{dt}\right)^{m}{\rm e}^{-\frac{|t|^{2}}{2}},
$$
where $t\in\mathbb R$, and
$$
c_{m}=2^{-m/2}(m!)^{-1/2}\pi^{-1/4}.
$$
For more details on the derivation of these formulae see e.g.  \cite{NiRo:10}.

Let us also calculate
\begin{align*}
\int_{\mathbb R^{n}}|(f\ast_{\H}
g)(x)|dx&\leq\int_{\mathbb R^{n}}\sum\limits_{\xi\in\ind}|\widehat{f}(\xi)\widehat{g}(\xi)|\,|e_{\xi}(x)|dx\\
&\leq\sum\limits_{\xi\in\ind}|\widehat{f}(\xi)|\,|\widehat{g}(\xi)|\, \|e_{\xi}\|_{L^{1}}\\
&\leq C\|f\|_{L^{2}}\,\|g\|_{L^{2}}\sup\limits_{\xi\in\ind}\|e_{\xi}\|_{L^{1}}\\
&\leq C\|f\|_{L^{2}}\,\|g\|_{L^{2}}.
\end{align*}
Hence we obtain the following estimate for the convolution:
if $f,g\in  L^{2}(\mathbb R^{n})$, then $f\ast_{\H} g\in L^{1}(\mathbb R^{n})$ with
\begin{equation}\label{EQ:ho-conv}
\|f \ast_{\H} g\|_{L^1}\leq C \|f\|_{L^2}\|g\|_{L^2},
\end{equation}
for some constant $C$.

\end{document}